\documentclass{amsart}
\usepackage{amsmath,amsthm,amssymb}
\usepackage{graphicx}
\usepackage{tikz-cd}
\usepackage{paralist}
\usepackage{environ}
\usepackage{xcolor}
\usepackage{hyperref}
\usepackage{centernot}
\usepackage{marvosym}

\usepackage[title]{appendix}

\usepackage{xcolor,cancel}

\makeatletter
\DeclareFontFamily{U}{tipa}{}
\DeclareFontShape{U}{tipa}{bx}{n}{<->tipabx10}{}
\newcommand{\arc@char}{{\usefont{U}{tipa}{bx}{n}\symbol{62}}}%

\newcommand{\arc}[1]{\mathpalette\arc@arc{#1}}

\newcommand{\arc@arc}[2]{%
  \sbox0{$\m@th#1#2$}%
  \vbox{
    \hbox{\resizebox{\wd0}{\height}{\arc@char}}
    \nointerlineskip
    \box0
  }%
}
\makeatother

\makeatletter
\newcommand{\doublewedge}{\big@doubleop{\wedge}}
\newcommand{\big@doubleop}[1]{%
  \DOTSB\mathop{\mathpalette\big@doubleop@aux{#1}}\slimits@
}

\newcommand\big@doubleop@aux[2]{%
  \sbox\z@{$\m@th#1#2$}%
  \makebox[1.35\wd\z@][s]{$\m@th#1#2\hss#2$}%
}

\newcommand{\abs}[1]{\left|#1\right|}     
\newcommand{\cl}{\mbox{cl}}  
\newcommand{\Int}{\mbox{int}} 
\newcommand{\vor}{\mbox{vor}} 
\newcommand{\cyc}{\mbox{cyc}} 

\newcommand{\near}{\delta} 
\newcommand{\dcap}{\mathop{\cap}\limits_{\Phi}} 
\newcommand{\dnear}{\delta_{\Phi}} 
\newcommand{\norm}[1]{\left\|#1\right\|}  
\newcommand{\assign}{\mathrel{\mathop :}=}

\usepackage{pstricks,pst-text,pst-grad,pst-node,pst-3dplot,pstricks-add,pst-poly,pst-coil,pst-bar,pst-all,pst-plot,pst-2dplot} 
\usepackage{pst-fun,pst-blur}
\usepackage{pst-all}
\usepackage{subfigure}
\renewcommand{\thesubfigure}{\thefigure.\arabic{subfigure}}
\makeatletter
\renewcommand{\p@subfigure}{}
\renewcommand{\@thesubfigure}{\thesubfigure:\hskip\subfiglabelskip}
\makeatother

\theoremstyle{plain}
\newtheorem{theorem}{Theorem}
\newtheorem{lemma}{Lemma}
\newtheorem{remark}{Remark}
\newtheorem{definition}{Definition}

\newtheorem{example}{Example}
\newtheorem{corollary}{Corollary}

\begin{document}

\title[Amenable Planar Vortexes]{Fixed Point Property of Amenable Planar Vortexes}

\author[J.F. Peters]{J.F. Peters}
\address{
Computational Intelligence Laboratory,
University of Manitoba, WPG, MB, R3T 5V6, Canada and
Department of Mathematics, Faculty of Arts and Sciences, Ad\.{i}yaman University, 02040 Ad\.{i}yaman, Turkey,
}
\email{james.peters3@umanitoba.ca}
\thanks{The research has been supported by the Natural Sciences \&
Engineering Research Council of Canada (NSERC) discovery grant 185986 
and Instituto Nazionale di Alta Matematica (INdAM) Francesco Severi, Gruppo Nazionale per le Strutture Algebriche, Geometriche e Loro Applicazioni grant 9 920160 000362, n.prot U 2016/000036 and Scientific and Technological Research Council of Turkey (T\"{U}B\.{I}TAK) Scientific Human
Resources Development (BIDEB) under grant no: 2221-1059B211301223.}
\author[T. Vergili]{T. Vergili}
\address{
Department of Mathematics, Karadeniz Technical University, Trabzon, Turkey,
}
\email{tanevergili@gmail.com}

\subjclass[2010]{37C25 (fixed point theory); 55M20 (fixed points); 54E05 (proximity); 55U10 (Simplicial sets and complexes)}

\date{}

\dedicatory{Dedicated to Mahlon M. Day}

\begin{abstract}
This article introduces free group representations of planar vortexes in a CW space that are a natural outcome of results for amenable groups and fixed points found by M.M. Day during the 1960s and a fundamental result for fixed points given by L.E.J. Brouwer.
\end{abstract}
\keywords{Amenable group, CW space, Fixed point, Planar vortex}
\maketitle
\tableofcontents

\section{Introduction}
This article introduces consequences of results for amenable groups and fixed points found by M.M. Day during the 1960s in terms of free group representations of planar vortexes in a CW space.  Results given here spring from a fundamental result for fixed points given by L.E.J. Brouwer~\cite{Brouwer1911fixedPoints}   

\begin{theorem}\label{thm:Brouwer}{\rm Brouwer Fixed Point Theorem~\cite[\S 4.7, p. 194]{Spanier1966AlgTopology}}$\mbox{}$\\
Every continuous map from $\mathbb{R}^n$ to itself has a fixed point.
\end{theorem}

Briefly, let $\Sigma$ be a finite group and let $m(\Sigma)$ be a set of bounded, real-valued functions on $\Sigma$.  Then $\Sigma$ is amenable, provided there is a mean $\mu$ on $m(\Sigma)$ which is both left and right invariant.

\begin{theorem}\label{thm:Day}{\rm Day Abelian Group [Semigroup] Theorem~\cite[p. 516]{Day1957amenableSemigroups,Day1961amenableSemigroups}}$\mbox{}$\\
Every finite group is amenable.
\end{theorem}

The study of amenable groups led to the following extension of the Kakutani-Markov Theorem by Day.

\begin{theorem}\label{thm:DayFixedPoints}{\rm Day Fixed Point Theorem~\cite[p.586]{Day1961amenableSemigroups}}.\\
Let $K$ be a compact convex subset of a locally convex linear topological space $X$, and let $\Sigma$ be a semigroup (under functional composition) of continuous affine transformations of $K$ into itself.  If \ $\Sigma$, when regarded as an abstract semigroup, is amenable, or if it has a left-invariant mean, then there is in $K$ a common fixed point of the family $\Sigma$.
\end{theorem}

A direct consequence of Theorem~\ref{thm:DayFixedPoints} is that each amenable group of a planar vortex has a fixed point in a CW space. 

\begin{definition}
A \emph{\bf planar vortex} $\vor E$ is a finite cell complex, which is a collection of path-connected vertices in nested, filled 1-cycles in a CW complex $K$.  A \emph{\bf 1-cycle} in $\vor E$ (denoted by $\cyc A$) is a sequence of edges with no end vertex and with a nonempty interior.  A geometric realization of $\vor E$ is denoted by $\abs{\vor E}$ on $\abs{K}$ in the Euclidean plane.
\end{definition}

A nonvoid collection of cell complexes $K$ is a \emph{Closure finite Weak} CW space, provided $K$ is Hausdorff (every pair of distinct cells is contained in disjoint neighbourhoods~\cite[\S 5.1, p. 94]{Naimpally2013}) and the collection of cell complexes in $K$ satisfy the Whitehead~\cite[pp. 315-317]{Whitehead1939homotopy}, ~\cite[\S 5, p. 223]{Whitehead1949BAMS-CWtopology} CW conditions, namely, the closure of each cell complex is in $K$ and the nonempty intersection of cell complexes is in $K$.

A number of important results concerning fixed points in this paper spring from \v{C}ech proximities, leading to descriptive proximally continuous maps.  A descriptive proximally continuous map is defined over descriptive \v{C}ech proximity spaces~\cite[\S 4.1]{DiConcilio2018MCSdescriptiveProximities} in which the description of a nonempty set is in the form of a feature vector derived from a probe function. 
For the details, see App.~\ref{appDetails}. 
  
\section{Conjugacy between proximal descriptively continuous maps}

This section introduces proximal conjugacy between two dynamical systems, which is an easy extension of topological conjugacy~\cite[\S 8.1,p. 243]{AdamsFranzosa:2007}.
Proximal conjugacy is akin to strongly amenable groups in which each of its proximal topological actions has a fixed point~\cite{FrishTamuzFerdowsi2019strongAmenability}.  Let $\sum$ denote either a semigroup or a group.  And let $m(\sum)$ be the set of bounded, real-valued functions $\theta$ on $\sum$ for which
\[
\norm{\theta} = lub_{x\in \sum}\abs{\theta(x)}.
\]
A {\bf mean} $\mu$ on $m(\sum)$ is an element of the $m(\sum)^*$ (in the conjugate space $B^*$  of a Banach space $B$ ~\cite[p.510]{Day1957amenableSemigroups}) such that, for each  $x\in m(\sum)$, we have
\[
glb_{x\in \sum}\theta(x) \leq \mu(x) \leq lub _{x\in \sum}\theta(x).
\]

\noindent An element of $\mu$ of  $m(\sum)^*$ is left[right] invariant, provided 
\[
\mu(\ell_\sigma x) = \mu(x) \ [\mu(r_\sigma x)= \mu(x)], \ 
\mbox{for all} 
\ x\in m(\sum), \sigma\in \sum, 
\]

\noindent where $(\ell_\sigma x)\sigma' = x(\sigma \sigma')$ and $(r_\sigma x)\sigma' = x(\sigma \sigma')$ for all $\sigma' \in \Sigma$.\\

\begin{definition}~\cite[p.515]{Day1957amenableSemigroups} 
A semigroup (also group) $\sum$ is amenable, provided there is a mean $\mu$ on $m(\sum)$, which is both left and right invariant.
\end{definition} 



\begin{theorem}\label{thm:amenableRibbonGroup}
A free group representation of a planar vortex in a CW space is amenable.
\end{theorem}
\begin{proof}
This result is a direct consequence of Theorem 2 from Day~\cite{Day1961amenableSemigroups} and (I)~\cite[p.516]{Day1957amenableSemigroups}, since, by construction, every free group $G$ representation of a planar vortex is in a CW space is finite and, by definition, an Abelian semigroup.
\end{proof}

Theorem~\ref{thm:amenableRibbonGroup} stems from Day's extension of Theorem~\ref{thm:KakutaniCommonFixedPoint} (restated by Day~\cite[p. 585]{Day1961amenableSemigroups}) to cover the case when the family in question is a semigroup.

\begin{theorem}\label{thm:KakutaniCommonFixedPoint}{\rm Kakutani-Markov Theorem~\cite{Kakutani1938twoFixedPointTheorems,Markov1936someTheorems}}
Let $K$ be a compact convex set in a locally convex linear topological space, and let $F$ be a commuting family of continuous, affine transformations, $f$, of $K$ into itself. Then there is a common fixed point of the functions in $F$; that is, there is an $x$ in $K$ such that $f(x) = x$ for every
$f$ in $F$.
\end{theorem}

 If a planar vortex $\vor E$ has no hole inside, then we no longer need to require it be convex. It is automatically convex compact hence we have the following. 
 
\begin{theorem}
	For a CW complex $K$, let $(K, \delta)$ be a proximity space that contains a planar vortex $\vor E$ without a planar hole and let  $f: \vor E \to \vor E$ be proximal continuous. Then $\vor E$ has a fixed point of $f$. 
\end{theorem}
\begin{proof}
	Since $\vor E$ is finite, the topology on the geometric realization $\abs{\vor E}$  of  $\vor E$  can be regarded as the subspace topology inherited from the Euclidean space $\mathbb{R}^2$. Then if we pass the geometric realization of $f$, denoted $\abs{f}$, we see that $\abs{f}$ is a continuos affine transformation.  This is true since $f$ maps two near subsets to the two near subsets.   Also notice that $\abs{\vor  E}$ is  convex compact subset of $\mathbb{R}^2$ and the collection of maps $\{\abs{f}^n \ : \ n=1,2\dotsc\}$ is amenable since it's a semigroup under composition. Then by  Theorem~\ref{thm:KakutaniCommonFixedPoint},  for the family of continuous affine transformations $\{\abs{f}^n \ : \ n=1,2,\dotsc\}$, there is an $x$ in $\abs{\vor E}$ such that $f(x)=x$ and so $\abs{\vor E}$ has a fixed point of $\abs{f}$. This also allow us to conclude  $\vor E$ has a fixed point of $f$ without considering the geometric realization.\\
\end{proof}

If a planar vortex $\vor E$ has a (planar) hole inside, then we could consider a subset $\vor E$ such  that its geometric realization is convex compact.

\begin{theorem}
		For a CW complex $K$, let $(K, \delta)$ be a proximity space that contains a planar vortex $\vor E$ with a planar hole and let $X\subset \vor E$ such that its geometric realization $\abs{X}$ is convex compact. If $f: X \to X$ is proximal continuous, then $X$ has a fixed point under $f$.	
\end{theorem}
\begin{proof}
	By the construction of a planar vortex, there is subset $X$ of $\vor E$ such that its geometric realization $\abs{X}$ is a convex compact subset of $\abs{\vor E}$. Then a proximal continuous map $f: X \to X$ has a corresponding continuous affine transformation $\abs{f}:  \abs{X} \to \abs{X}$. Again by Theorem~\ref{thm:KakutaniCommonFixedPoint},  for the family of continuous affine transformations $\{\abs{f}^n \ : \ n=1,2,\dotsc\}$, there is an $x$ in $\abs{X}$ such that $f(x)=x$ and so $\abs{X}$ has a fixed point of $\abs{f}$. This concludes that 
	$\vor E$ has a fixed point of $f$ without considering the geometric realization.
\end{proof}


\begin{remark}\label{rem:vortexStructure}
From what have observed, notice that 
any vortex has a locally compact abelian group representation, since it is a locally compact Hausdorff space  and its underlying group structure is abelian. In that case, any proximal continuous map from a vortex to itself can be also considered as a group action. Hence, by a  direct consequence of Theorem~\ref{thm:amenableRibbonGroup} and a result of a generalization of the Kakutani-Markov Theorem~\ref{thm:KakutaniCommonFixedPoint}, each amenable vortex has a fixed point. \qquad\textcolor{blue}{\Squaresteel}
\end{remark}

\begin{corollary}
	If $f$ is a proximal continuous map from a vortex to itself, then $f$ has a fixed point.
\end{corollary}


Next, we introduce the (descriptive) proximal conjugate   between two proximal  (descriptive) continuous maps. Note that a (descriptive) Cech proximity space $X$ together with a (descriptive) proximal continuous self map on $X$ can be considered as a (descriptive) proximal dynamical system. Now we introduce a (descriptive) proximal conjugacy between two (descriptive) dynamical systems, so that  the existence of it guarantees the (descriptive) dynamical systems having equivalent flows and related (descriptive) fixed points.\\

\begin{definition}
Two  proximal continuous maps $f: (X,\delta_1) \to (X,\delta_1) $ and $g: (Y,\delta_2) \to (Y,\delta_2)$ are said to be proximal  conjugates, provided there exists a proximal  isomorphism $h: (X,\delta_1)  \to (Y,\delta_2)$ such that $g\circ h = h\circ f$. The function $h$ is called a proximal conjugacy between $f$ and $g$.
\end{definition}

The following theorem states that if two proximal continuous maps are proximal conjugate, then their corresponding iterated functions are also proximal conjugate.

\begin{theorem}
	Let $h$ be a proximal conjugacy between $f: (X,\delta_1) \to (X,\delta_1) $ and $g: (Y,\delta_2) \to (Y,\delta_2)$. Then for each $A \subseteq X$ and $n\in \mathbb{Z}_+$, we have $h(f^n(A))=g^n(h(A))$.  
\end{theorem}
\begin{proof}
	The proof follows from the induction on $n$.
\end{proof}

\begin{definition}\label{def:proximalDescriptiveContinuousMaps}
Two  proximal descriptive continuous maps $f: (X,\delta_{\Phi_1}) \to (X,\delta_{\Phi_1}) $ and $g: (Y,\delta_{\Phi_2}) \to (Y,\delta_{\Phi_2})$ are said to be proximal descriptive conjugates, provided there exists a proximal descriptive isomorphism $h: (X,\delta_{\Phi_1})  \to (Y,\delta_{\Phi_2})$ such that $g\circ h(A) \underset{\mbox{des}}{=} h\circ f(A)$ for any $A \in 2^X$. The function $h$ is called a proximal descriptive conjugacy between $f$ and $g$. 
\end{definition}

\begin{remark}
	We see from the definition of a proximal descriptive conjugacy that $g\circ h(A)$ and $h\circ f(A)$ may not be equal but we have 
	\[ \Phi_2(g\circ h(A))= \Phi_2(h\circ f(A)) \]
	for $A\in 2^X$. Moreover $g\circ h(A) \underset{\mbox{des}}{=} h\circ f(A)$ implies  $g(A) \underset{\mbox{des}}{=} h\circ f\circ h^{-1}(A)$  and  $f(A) \underset{\mbox{des}}{=} h^{-1}\circ g\circ h(A)$, so that we have the following commutative diagrams.  
\end{remark}

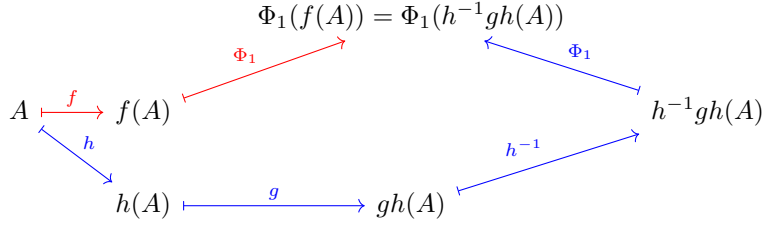
\begin{tikzcd}
	&  & \Phi_1(f(A))=\Phi_1(h^{-1}gh(A))  & \\
	A \arrow[r, mapsto, red, "f"]  \arrow[dr, mapsto, blue, "h"] 
	& f(A)   \arrow[ur, mapsto, red, "\Phi_1"]  & &   h^{-1}gh(A) 
	\arrow[ul, mapsto, blue, "\Phi_1"']  \\ 
	& h(A) \arrow[r, mapsto, blue, "g"]
	&gh(A)     \arrow[ur, mapsto, blue, "h^{-1}"] & 
\end{tikzcd}

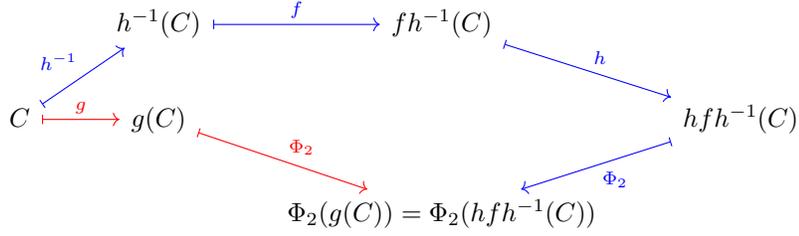
\vspace*{1cm}

\begin{tikzcd}
	 & h^{-1}(C)  \arrow[r, mapsto, blue, "f"]  & fh^{-1}(C)    \arrow[dr, mapsto, blue, "h"] & \\
	 C   \arrow[ur, mapsto, blue, "h^{-1}"] \arrow[r, mapsto, red, "g"] & g(C)    \arrow[dr, mapsto, red, "\Phi_2"]   &  & hfh^{-1}(C)   \arrow[dl, mapsto, blue, "\Phi_2"] \\
	&  & \Phi_2(g(C))=\Phi_2(hfh^{-1}(C))   & 
\end{tikzcd}

\begin{remark}
For proximal descriptive conjugates $f: (X,\delta_{\Phi_1}) \to (X,\delta_{\Phi_1}) $ and $g: (Y,\delta_{\Phi_2}) \to (Y,\delta_{\Phi_2})$, Def.~\ref{def:proximalDescriptiveContinuousMaps} tells us that for $A\subseteq X$ and $C\subseteq Y$, we have 
\begin{align*} 
\Phi_2(g\circ h(A)) &=  \Phi_2(h\circ f(A)),  \\
\Phi_1(f\circ h^{-1}(C)) &= \Phi_1(h^{-1}\circ g(C)).\qquad \mbox{\textcolor{blue}{\Squaresteel}}
\end{align*}
\end{remark}

Note that  if  $h$ is a proximal descriptive conjugacy between $f: (X,\delta_{\Phi_1}) \to (X,\delta_{\Phi_1}) $ and $g: (Y,\delta_{\Phi_2}) \to (Y,\delta_{\Phi_2})$, then $A \ \underset{\mbox{des}}{=} \  B$  implies $h(A) \ \underset{\mbox{des}}{=} \  h(B)$ for $A, B \in 2^X$.

\begin{theorem} \label{thm:desconj}
Let $h$ be a proximal descriptive conjugacy between $f: (X,\delta_{\Phi_1}) \to (X,\delta_{\Phi_1}) $ and $g: (Y,\delta_{\Phi_2}) \to (Y,\delta_{\Phi_2})$. Then for each $A \in 2^X$ and $n\in \mathbb{Z}_+$, we have $h(f^n(A)) \ \underset{\mbox{des}}{=} \   g^n(h(A))$.  
\end{theorem}
\begin{proof}
The proof follows from the induction on $n$.
\end{proof}

\begin{corollary}\label{cor:proxDescripConjugacy}
		Let $h$ be a proximal descriptive conjugacy between $f: (X,\delta_{\Phi_1}) \to (X,\delta_{\Phi_1}) $ and $g: (Y,\delta_{\Phi_2}) \to (Y,\delta_{\Phi_2})$.
		
		\begin{compactenum}[a)]
			\item If $A$ is a descriptively fixed subset of $f$, then $h(A)$ is a descriptively fixed subset of $g$.
			\item If $A$ is an eventual descriptively fixed subset of $f$, then $h(A)$ is an eventual  descriptively fixed subset of $g$.
			\item If $A$ is an almost descriptively fixed subset of $f$, then $h(A)$ is an almost  descriptively fixed subset of $g$.
		\end{compactenum}
\end{corollary}
\begin{proof}
	\begin{compactenum}[a)]
		\item  Let $A$ be a descriptively fixed subset of $f$. That is,  $\Phi_1(f(A))= \Phi_1(A)$. In other words, we have $f(A) \ \underset{\mbox{des}}{=} \  A$. 
		Since $h$ is a proximal isomorhism,  $h$ preserves desciptive proximity  $h(f(A)) \ \underset{\mbox{des}}{=} \  h(A)$.  By Theorem~\ref{thm:desconj}, $g(h(A)) \ \underset{\mbox{des}}{=} \   h(A)$ so that $h(A)$ is a  descriptively fixed subset of $g$. \\ 
		\item Let $A$ be an eventual descriptively fixed subset of $f$. That is, $A$ is not a descriptively fixed subset of $f$ but  $\Phi_1(f^n(A))= \Phi_1(A)$ for some positive integer $n>1$. In other words, we have $f^n(A) \ \underset{\mbox{des}}{=} \  A$. 
		Since $h$ is a proximal isomorhism,  $h$ preserves being equal  in a descriptive sense:  $h(f^n(A)) \ \underset{\mbox{des}}{=} \  h(A)$. By Theorem~\ref{thm:desconj}, $g^n(h(A)) \ \underset{\mbox{des}}{=} \  h(A)$. 
		Note that $h(A)$ is not a descriptively fixed subset of $g$ since  $A$ is not a descriptively fixed subset of $f$  and $h$ is an isomorphism. So, $h(A)$ is an eventual descriptively fixed subset of $g$. \\ 
		\item Let $A$ be  an almost descriptively fixed subset of $f$. That is, $f(A) \ \underset{\mbox{des}}{=} \  A$ or     $A \ \delta_{\Phi_1}  \ f(A)$. If $f(A) \ \underset{\mbox{des}}{=} \  A$, then  we are done. Let $A \ \delta_{\Phi_1}  \ f(A)$. Since $h$  is a proximal isomorphism, we have $h(A) \ \delta_{\Phi_2}  \ h(f(A))$. By Theorem~\ref{thm:desconj},  $h(A) \ \delta_{\Phi_2} \ g(h(A))$ so that $h(A)$ is a descriptively fixed subset of $g$.
	\end{compactenum}
\end{proof}

Further, the existence of  proximal conjugacy between two dynamical systems of cell complexes such as vortexes also guarantees isomorphic amenable group structures and hence related fixed points, which is another consequence of Theorem~\ref{thm:amenableRibbonGroup}.



\begin{corollary}\label{cor:dpConjugacy}
	If there exists a descriptive proximal conjugacy between two descriptive dynamical systems, then they have isomorphic descriptive fixed subsets. 
\end{corollary}
\begin{proof}
	Let $f: (X, \delta_{\Phi_1}) \to (X,\delta_{\Phi_1})$ and $g: (Y, \delta_{\Phi_2}) \to (Y, \delta_{\Phi_2})$ be proximal descriptive conjugates and $h : (X, \delta_{\Phi_1}) \to (Y, \delta_{\Phi_2})$ be the proximal descriptive conjugacy between them. If $A \in 2^X$ is a descriptive fixed subset of $f$, then $h(A)$ is a descriptive fixed subset of $g$ by Corollary~\ref{cor:proxDescripConjugacy} so that $A$ and $h(A)$ are descriptively isomorphic. Similarly if $B \in 2^Y$ is a descriptive fixed subset of $g$, then $h^{-1}(B)$ is a descriptive fixed subset of $f$ so that $B$ and $h^{-1}(B)$ are descriptively isomorphic. Hence there is a one-to-one correspondence between the set of the descriptive fixed subsets of $f$ and  the set of the descriptive fixed subsets of $g$. 
\end{proof}

\section{Weak conjugacy between descriptive proximally continuous maps}
This section introduces weak conjugacy between  descriptive proximally continuous maps.

\begin{definition}\label{def:weakProxConjugacy}
	Two  proximally continuous maps $f: (X,\delta_1) \to (X,\delta_1) $ and $g: (Y,\delta_2) \to (Y,\delta_2)$ are said to be weakly proximal conjugates, provided there exists a proximal  isomorphism $h:(X,\delta_1) \to (Y,\delta_2)$ such that for any $A\in 2^X$,  $g\circ h(A)\ \delta_2\ h\circ f(A)$. Note that this also implies that $f\circ h^{-1}(C)\ \delta_1  \ h^{-1}\circ g(C)$ for any $C\in 2^Y$.
	The function $h$ is called a  weakly proximal conjugacy between $f$ and $g$. 
\end{definition}

\begin{theorem}
Let $h$ be a weakly proximal conjugacy between $f: (X,\delta_1) \to (X,\delta_1) $ and $g: (Y,\delta_2) \to (Y,\delta_2)$. Then for each $A\in 2^X$ and $n\in \mathbb{Z}_+$, we have\\ 
$h(f^n(A)) \ \delta_2 \ g^n(h(A))$.  
\end{theorem}
\begin{proof}
The proof follows from the induction on $n$.
\end{proof}

\begin{definition}\label{def:weakProxDescrConjugates}
Two descriptive proximally continuous maps $f: (X,\delta_{\Phi_1}) \to (X,\delta_{\Phi_1}) $ and $g: (Y,\delta_{\Phi_2}) \to (Y,\delta_{\Phi_2})$ are said to be weakly proximal descriptive conjugates, provided there exists a proximal descriptive isomorphism $h: (X,\delta_{\Phi_1})  \to (Y,\delta_{\Phi_2})$ such that $g\circ h(A)\ \delta_{\Phi_2}\ h\circ f(A)$ for any $A \in 2^X$. Note that this also implies  $f\circ h^{-1}(C)\ \delta_{\Phi_2}\ h^{-1}\circ g(C)$ for any $C\in 2^Y$. The function $h$ is called a weakly proximal descriptive conjugacy between $f$ and $g$. 
\end{definition}

\begin{remark}
For weakly proximal descriptive conjugates $f: (X,\delta_{\Phi_1}) \to (X,\delta_{\Phi_1}) $ and $g: (Y,\delta_{\Phi_2}) \to (Y,\delta_{\Phi_2})$, Def.~\ref{def:weakProxDescrConjugates} and Lemma~\ref{lemma:dP2converse} tell us that for $A\in 2^X$ and $C\in 2^Y$, we have 
\begin{align*} 
g\circ h(A)\ &\dcap\  f\circ h(A)\neq \emptyset, \\
f\circ h^{-1}(C)\ &\dcap\ h^{-1}\circ g(C) \neq \emptyset..\qquad \mbox{\textcolor{blue}{\Squaresteel}}
\end{align*}
\end{remark}

\begin{appendices}\label{appDetails}
\section{Planar Vortexes}\label{sec:vortexes}
This section briefly looks at planar vortex structures in planar CW spaces.  
For simplicity, we consider only 2 cycle vortexes containing a pair of nested 1-cycles that intersect or attached to each other via at least one bridge edge.
%

\begin{definition} Planar 2 Cycle Vortex \label{def:ribbon}{\rm \cite{Peters2020AMSBullRibbonComplexes}}.\\
Let $\cyc A, \cyc B$ be a collection of path-connected vertexes on nested filled 1-cycles (with $\cyc B$ in the interior of $\cyc A$) defined on a finite, bounded, planar region in a CW space $K$.  A \emph{planar 2 cycle vortex} $\vor E$ is defined by
\[
\vor E = \overbrace{
\left\{\cl(\cyc A): \cl(\cyc B)\subset \Int(\cl(\cyc A))\right\}.}^{\mbox{\textcolor{blue}{\bf $\cl(\cyc B)$ is contained (nested) in the interior of $\cl(\cyc A)$.}}}\mbox{\textcolor{blue}{\Squaresteel}}
\]
\end{definition}

A vortex containing adjacent non-intersecting cycles has a bridge edge attached to vertexes on the cycles.

\begin{definition}
A vortex {\bf bridge edge} is an edge attached to vertexes on a pair of non-interecting, filled 1-cycles.
\end{definition}

\begin{figure}[!ht]
\centering
\subfigure[Vortex $\abs{\vor E}$ with non-intersecting 1-cycles $\cyc A, \cyc B$]
 {\label{fig:rbE}
\begin{pspicture}
(-1.5,-0.5)(4.0,3.0)
\psframe[linewidth=0.75pt,linearc=0.25,cornersize=absolute,linecolor=blue](-1.25,-0.25)(3.25,3)
\psline*[linestyle=solid,linecolor=green!30]%
(0,0)(1,0.5)(2.0,0.0)(3.0,0.5)(3.0,1.5)(2.0,2.0)(1,1.5)(0,2)
(-1,1.5)(-1,0.5)(0,0)
\psline[linestyle=solid,linecolor=black]%
(0,0)(1,0.5)(2.0,0.0)(3.0,0.5)(3.0,1.5)(2.0,2.0)(1,1.5)(0,2)(-1,1.5)(-1,0.5)(0,0)
\psdots[dotstyle=o,dotsize=2.5pt,linewidth=1.2pt,linecolor=black,fillcolor=white]%
(0,0)(1,0.5)(2.0,0.0)(3.0,0.5)(3.0,1.5)(2.0,2.0)(1,1.5)(0,2)(-1,1.5)(-1,0.5)(0,0)
\psline[linestyle=solid](1,1.5)(2,2)\psline[arrows=<->](-1,1.7)(-0.3,2.0)\psline[arrows=<->](0.2,2.05)(0.8,1.75)
\psline*[linestyle=solid,linecolor=gray!20]%
(0,0.25)(1,0.75)(2.0,0.25)(2.5,0.5)(2.5,0.75)(2.0,1.35)(1,1.25)(0,1.5)(-.55,1.25)(-.55,0.75)(0,0.25)
\psline[linestyle=solid,linecolor=black]%
(0,0.25)(1,0.75)(2.0,0.25)(2.5,0.5)(2.5,0.75)(2.0,1.35)(1,1.25)(0,1.5)(-.55,1.25)(-.55,0.75)(0,0.25)
\psline[linestyle=solid,linecolor=black]
(2.0,1.35)(2.0,2.0)
\psdots[dotstyle=o,dotsize=2.5pt,linewidth=1.2pt,linecolor=black,fillcolor=white]%
(0,0.25)(1,0.75)(2.0,0.25)(2.5,0.5)(2.5,0.75)(2.0,1.35)(1,1.25)(0,1.5)(-.55,1.25)(-.55,0.75)(0,0.25)
\psdots[dotstyle=o,dotsize=2.5pt,linewidth=1.2pt,linecolor=black,fillcolor=red]
(2.0,1.35)(2.0,2.0)
\rput(-1.0,2.75){\footnotesize $\boldsymbol{K}$}
\rput(0.0,2.25){\footnotesize $\boldsymbol{\abs{\vor E}}$}
\rput(2.8,1.85){\footnotesize $\boldsymbol{cyc A}$}
\rput(2.5,1.25){\footnotesize $\boldsymbol{cyc B}$}
\rput(2.0,2.15){\footnotesize $\boldsymbol{p}$}
\rput(1.9,1.15){\footnotesize $\boldsymbol{q}$}
\end{pspicture}}\hfil
\subfigure[Vortex $\abs{\vor E'}$ with intersecting 1-cycles $\cyc A',\cyc B'$]
 {\label{fig:rbNrvE}
\begin{pspicture}
(-1.5,-0.5)(4.0,3.0)
\psframe[linewidth=0.75pt,linearc=0.25,cornersize=absolute,linecolor=blue](-1.45,-0.25)(3.25,3)
\psline*[linestyle=solid,linecolor=green!30]%
(0,0)(1,0.5)(2.0,0.0)(3.0,0.5)(3.0,1.5)(2.0,2.0)(1,1.5)(0,2)(-1,1.5)(-1,0.5)(0,0)
\psline[linestyle=solid,linecolor=black]%
(0,0)(1,0.5)(2.0,0.0)(3.0,0.5)(3.0,1.5)(2.0,2.0)(1,1.5)(0,2)(-1,1.5)(-1,0.5)(0,0)
\psdots[dotstyle=o,dotsize=2.5pt,linewidth=1.2pt,linecolor=black,fillcolor=white]%
(0,0)(1,0.5)(2.0,0.0)(3.0,0.5)(3.0,1.5)(2.0,2.0)(1,1.5)(0,2)(-1,1.5)(-1,0.5)(0,0)
\psline[arrows=<->](-1,1.7)(-0.3,2.0)
\psline[arrows=<->](0.2,2.05)(0.8,1.75)
\psline*[linestyle=solid,linecolor=gray!20]%
(0,0.25)(1,0.5)(2.0,0.25)(2.5,0.5)(2.5,0.75)(2.0,2.0)(1,1.25)(0,1.5)(-.55,1.25)(-1,1.5)(0,0.25)
\psline[linestyle=solid,linecolor=black]%
(0,0.25)(1,0.5)(2.0,0.25)(2.5,0.5)(2.5,0.75)(2.0,2.0)(1,1.25)(0,1.5)(-.55,1.25)(-1,1.5)(0,0.25) 
\psdots[dotstyle=o,dotsize=2.5pt,linewidth=1.2pt,linecolor=black,fillcolor=white]%
(0,0.25)(1,0.5)(2.0,0.25)(2.5,0.5)(2.5,0.75)(2.0,2.0)(1,1.25)(0,1.5)(-.55,1.25)(-1,1.5)(0,0.25)
\psdots[dotstyle=o,dotsize=2.5pt,linewidth=1.2pt,linecolor=black,fillcolor=red!80]
(-1,1.5)(2.0,2.0)(1,0.5)
\rput(-1.0,2.75){\footnotesize $\boldsymbol{K'}$}
\rput(-1.2,1.6){\footnotesize $\boldsymbol{g_1}$}
\rput(2.0,2.15){\footnotesize $\boldsymbol{g_2}$}
\rput(1,0.65){\footnotesize $\boldsymbol{g_3}$}
\rput(2,0.45){\footnotesize $\boldsymbol{b}$}
\rput(0.0,2.25){\footnotesize $\boldsymbol{\abs{\vor E'}}$}
\rput(2.8,1.85){\footnotesize $\boldsymbol{cyc A'}$}
\rput(2.5,1.25){\footnotesize $\boldsymbol{cyc B'}$}
l\end{pspicture}}
\caption[]{Sample planar 2-cycle vortexes}
\label{fig:2Ribbons}
\end{figure}

\begin{remark}
From Def.~\ref{def:ribbon}, the cycles in a 2 cycle vortex can either have nonempty intersection (see, {\em e.g.}, $\cyc A'\cap\cyc B'\neq \emptyset$ in $\abs{\vor E'}$ in Fig.~\ref{fig:rbNrvE}) or there is a bridge edge between the cycles (see, {\em e.g.}, $\arc{pq}$ $\abs{\vor E}$  in Fig.~\ref{fig:rbE}). In effect, every pair of vertexes in a 2 cycle vortex is path-connected.  
\textcolor{blue}{\Squaresteel}
\end{remark}


\begin{remark}
The structure of a 2 cycle vortex extends to a vortex with $k > 2$ nested filled 1-cycles, provided adjacent pairs of cycles $\cyc A,\cyc A'$ in a $k$-cycle vortex either intersect or there is a bridge edge attached between $\cyc A,\cyc A'$.
\textcolor{blue}{\Squaresteel}
\end{remark}

\section{Free Group Representation of a Vortex}

A finite group $G$ is free, provided every element $x\in G$ is
a linear combination of its basis elements (called generators).
We write $\mathcal{B}$ to denote a nonempty set of generators $\left\{g_1,\dots,g_{\abs{\mathcal{B}}}\right\}$ and $G(\mathcal{B},+)$ to denote the free group with binary operation $+$.

\begin{example}
The basis $\left\{g_1,g_2,g_3\right\}$ generates a group $G$ whose geometric realization is $\abs{\vor E'}$ in Fig.~\ref{fig:rbNrvE}. The $+$ operation on $G$ corresponds to a move from a generator to a neighbouring vertex. For example,
\begin{align*}
b &= \overbrace{3g_1 + 4g_2}^{\mbox{\textcolor{blue}{\bf traversing 3 $\boldsymbol{\cyc A'}$ \& 4 $\boldsymbol{\cyc B'}$\ edges to reach $\boldsymbol{b}$ via $\boldsymbol{g_1,g_3}$}}}\\
  &= \overbrace{2g_1 + 1g_3}^{\mbox{\textcolor{blue}{\bf traversing 3 $\boldsymbol{\cyc B'}$ edges to reach $\boldsymbol{b}$ via $\boldsymbol{g_1,g_2}$}}}\\
	&= \overbrace{0g_1 + 1g_3 = 1g_3.}^{\mbox{\textcolor{blue}{\bf traversing 1 $\boldsymbol{\cyc B'}$ edge to reach $\boldsymbol{b}$ via $\boldsymbol{g_3}$}}}
\end{align*}

The identity element 0 in $G$ is represented by a zero move from a generator $g$ to another vertex (denoted by $0g$) and an inverse in $G$ is represented by a reverse move $-g$.
\quad\textcolor{blue}{\Squaresteel}
\end{example}

\begin{definition}
Let $2^K$ be the collection of cell complexes in a CW space $K$, vortex $\abs{\vor E}\in 2^K$, basis $\mathcal{B}\in\abs{\vor E}$, $k_i$ the $i^{th}$ integer coeficient in a linear combination $\mathop{\sum}\limits_{i,j}k_ig_j$ of generating elements $g_j\in \mathcal{B}$.  A free group $G$ {\bf representation} of $\abs{\vor E}$ is a continuous self-map $f:2^K\to 2^K$ defined by
\begin{align*}
f(\abs{\vor E}) &= \left\{v \assign{\mathop{\sum}\limits_{i,j}k_ig_j}: v\in \abs{\vor E},g_j\in \mathcal{B}\right\}\\
          &= \overbrace{\boldsymbol{G(\left\{ g_1,\dots,g_{_{\abs{\tiny \mathcal{B}}}}\right\}},+).}^{\mbox{\textcolor{blue}{\bf $\boldsymbol{\abs{\vor E}}\mapsto$ free group $\boldsymbol{G}$}}}
\end{align*}
\end{definition}


\section{Descriptive Proximity Spaces}\label{sec:Cech}
This section briefly introduces descriptive \v{C}ech proximity spaces, paving the way for descriptive proximally continuous maps. 
The simplest form of proximity relation (denoted by $\delta$) on a nonempty set was intoduced by E.\v{C}ech~\cite{Cech1966}.  A nonempty set $X$ equipped with the relation $\near$ is a \v{C}ech proximity space (denoted by $(X,\near$)), provided the following axioms are satisfied.\\
\vspace{3mm}

\noindent {\bf \v{C}ech Axioms}

\begin{description}
	\item[({\bf P}.0)] All nonempty subsets in $X$ are far from the empty set, {\em i.e.}, $A\ \not{\near}\ \emptyset$ for all $A\subseteq X$.
\item[({\bf P}.1)] $A\ \near\ B \Rightarrow B\ \near\ A$.
\item[({\bf P}.2)] $A\ \cap\ B\neq \emptyset \Rightarrow A\ \near\ B$.
\item[({\bf P}.3)] $A\ \near\ \left(B\cup C\right) \Rightarrow A\ \near\ B$ or $A\ \near\ C$.
\end{description}

 Given that a nonempty set $E$ has $k \geq 1$ features such as Fermi energy $E_{Fe}$, cardinality $E_{card}$, a description $\Phi(E)$ of $E$ is a feature vector, {\em i.e.}, $\Phi(E) = \left(E_{Fe},E_{card}\right)$. Nonempty sets $A,B$ with overlapping descriptions are descriptively proximal (denoted by $A\ \dnear\ B$). 
  The descriptive intersection of nonempty subsets in $A\cup B$ (denoted by $A\ \dcap\ B$) is defined by
\[
A\ \dcap\ B = \overbrace{\left\{x\in A\cup B: \Phi(x) \in \Phi(A)\ \cap\ \Phi(B)\right\}.}^{\mbox{\textcolor{blue}{\bf {\em i.e.}, $\boldsymbol{\mbox{Descriptions}\ \Phi(A)\ \&\ \Phi(B)\ \mbox{overlap}}$}}}
\] 

Let $2^X$ denote the collection of all subsets in a nonvoid set $X$. A nonempty set $X$ equipped with the relation $\dnear$ with non-void subsets $A,B,C\in 2^X$ is a descriptive proximity space, provided the following descriptive forms of the \v{C}ech axioms are satisfied.\\
\vspace{3mm}

\noindent {\bf Descriptive \v{C}ech Axioms}

\begin{description}
\item[({\bf dP}.0)] All nonempty subsets in $2^X$ are descriptively far from the empty set, {\em i.e.}, $A\ \not{\dnear}\ \emptyset$ for all $A\in 2^X$.
\item[({\bf dP}.1)] $A\ \dnear\ B \Rightarrow B\ \dnear\ A$.
\item[({\bf dP}.2)] $A\ \dcap\ B\neq \emptyset \Rightarrow A\ \dnear\ B$.
\item[({\bf dP}.3)] $A\ \dnear\ \left(B\cup C\right) \Rightarrow A\ \dnear\ B$ or $A\ \dnear\ C$.
\end{description}

\noindent The converse of Axiom ({\bf dp}.2) also holds.
\begin{lemma}\label{lemma:dP2converse}{\rm \cite{Peters2019vortexNerves}}
Let $X$ be equipped with the relation $\dnear$, $A,B\in 2^X$.   Then $A\ \dnear\ B$ implies $A\ \dcap\ B\neq \emptyset$.
\end{lemma}
\begin{proof}
Let $A,B\in 2^X$. By definition, $A\ \dnear\ B$ implies that there is at least one member $x\in A$ and $y\in B$ so that $\Phi(x) = \Phi(y)$, {\em i.e.}, $x$ and $y$ have the same description.  Then $x,y\in A\ \dcap\ B$.
Hence, $A\ \dcap\ B\neq \emptyset$, which is the converse of ({\bf dp}.2).    
\end{proof}

\begin{theorem}\label{theorem:dP2result}
Let $K$ be a cell complex, $Vor(K)\subset K$ a collection of planar vortexes equipped with the proximity $\dnear$ and let $\vor A,\vor B\in Vor(K)$. Then $\vor A\ \dnear\ \vor B$ implies $\vor A\ \dcap\ \vor B\neq \emptyset$.
\end{theorem}
\begin{proof}
Immediate from Lemma~\ref{lemma:dP2converse}. 
\end{proof}

Let $(X,\delta_1)$ and $(Y,\delta_2)$ be two \v{C}ech proximity spaces. Then a map $f: (X,\delta_1)\to (Y,\delta_2)$ is \emph{proximal continuous}, provided  $A\ \delta_1\ B$ implies $f(A)\ \delta_2\ f(B)$, \emph{i.e.},\\ $f(A)\ \delta_2\ f(B)$, provided $f(A)\ \cap\ f(B)\neq \emptyset$ for $A, B\in 2^X$~\cite[\S 1.4]{Naimpally70}.  In general, a proximal continuous function preserves the nearness of pairs of sets~\cite[\S 1.7,p. 16]{Naimpally2013}. Further, $f$ is a \emph{proximal isomorphism}, provided $f$ is proximal continuous with a proximal continuous inverse $f^{-1}$.

Let $(X,\delta_{\Phi_1})$ and $(Y,\delta_{\Phi_2})$  be  descriptive proximity spaces with  probe functions $\Phi_1: X \to \mathbb{R}^n$, $\Phi_2: Y \to \mathbb{R}^n$, and  $A,B \in 2^X$. Then a map $f : (X,\delta_{\Phi_1})\to (Y,\delta_{\Phi_2})$ is said to be \emph{descriptive proximally continuous}, provided $A\ \delta_{\Phi_1}\ B$ implies $f(A)\ \delta_{\Phi_2}\ f(B)$, \emph{i.e.}, $f(A)\ \delta_{\Phi_2} \ f(B)$, provided $f(A)\ \dcap \ f(B)\neq \emptyset$. Further $f$ is a \emph{descriptive proximal  isomorphism}, provided $f$ and its inverse $f^{-1}$  are  descriptively proximally continuous. \\

\begin{definition}
	Let $(X,\delta_{\Phi})$ be a descriptive  \v{C}ech proximity space and $f: (X, \delta_{\Phi}) \to (X, \delta_{\Phi})$ a descriptive proximally continuous map. A set $A \in 2^X$ is said to be descriptively invariant with respect to $f$, provided $\Phi(f(A))\subseteq \Phi(A)$. 
\end{definition}

Notice that if $A$ is a descriptively invariant set with respect to $f$, then $\Phi(f^n(A))\subseteq \Phi(A)$ for all positive integer $n$.

\begin{theorem}
	Let $(X, \delta_{\Phi})$ be a descriptive \v{C}ech proximity space and  $f: (X, \delta_{\Phi}) \to (X, \delta_{\Phi})$ a proximal descriptive continuous map. If $\{ A_i\}_{i \in I} \subseteq 2^X$ is a collection of descriptively invariant sets with respect to $f$, then
	\begin{compactenum}
		\item[i)] 	$\cup_{i\in I} A_i$ is   descriptively invariant with respect to  $f$, and 
		\item[ii)] $\cap_{i\in I} A_i$ is   descriptively invariant  with respect to  $f$. 
	\end{compactenum}
\end{theorem}
\begin{proof}
	From our assumption, we have $\Phi(f(A_i))\subseteq \Phi(A_i)$ for all $i\in I$ so that
	
	\begin{compactenum}
		\item[i)] \begin{align*}
		f(\cup_{i\in I} A_i) &= \cup_{i\in I} f(A_i) \\
		\Phi(f(\cup_{i\in I} A_i)) & = \Phi(\cup_{i\in I} f(A_i)) \\
		&=\cup_{i\in I} \Phi (f(A_i)) \\
		& \subseteq \cup_{i\in I} \Phi(A_i)
		\end{align*}
		and 
		
		\item[ii)] \begin{align*}
		f(\cap_{i\in I} A_i) &\subseteq  \cap_{i\in I} f(A_i) \\
		\Phi(f(\cap_{i\in I} A_i)) &\subseteq \Phi(\cap_{i\in I} f(A_i)) \\
		&\subseteq \cap_{i\in I} \Phi (f(A_i)) \\
		& \subseteq \cap_{i\in I} \Phi(A_i)
		\end{align*}
	\end{compactenum}
\end{proof}

\begin{theorem}
	Let $(X, \delta_{\Phi})$ be a descriptive \v{C}ech proximity space and  $f: (X, \delta_{\Phi}) \to (X, \delta_{\Phi})$ a descriptive proximally continuous map. If $A \in 2^X$ is  descriptively invariant with respect to $f$ then $\cl_{\delta_{\Phi}} A$ is also  descriptively invariant with respect to with respect to $f$.
\end{theorem}
\begin{proof}
The descriptive closure of a subset $A$ of $X$ is defined in~\cite[\S 1.21.2]{Peters2013springer} as follows: 	
\[ 
\cl_{\Phi} A=\{ x\in X \ | \ x \ \delta_{\Phi} \ A \}.
\]
Take  an element $x$ in  $ \cl_{\Phi} A$ so that  $x\ \delta_{\Phi} \ A$ and  $\Phi(x) \in \Phi(A)$ by Lemma~\ref{lemma:dP2converse}.     Since $f$ is a descripitive proximally continuous $f(x) \ \delta_{\Phi} \ f(A)$ and $\Phi(f(x))\in \Phi(f(A))$ by Lemma~\ref{lemma:dP2converse}. We also have $\Phi(f(x))\in \Phi(A)$ since $A$ is an invariant set with respect to $f$. Therefore  $f(x)\ \delta_{\Phi} \ A$  and   $f(x) \in \cl_{\Phi} A$. Since this holds for all $x \in \cl_{\Phi}A$, we have $f(\cl_{\Phi} A) \subseteq \cl_{\Phi} A$ so that         $\Phi(f(\cl_{\Phi} A)) \subseteq \Phi(\cl_{\Phi} A)$.
\end{proof}

\end{appendices}

\bibliographystyle{amsplain}
\bibliography{NSrefs}

\providecommand{\bysame}{\leavevmode\hbox to3em{\hrulefill}\thinspace}
\providecommand{\MR}{\relax\ifhmode\unskip\space\fi MR }
\providecommand{\MRhref}[2]{%
  \href{http://www.ams.org/mathscinet-getitem?mr=#1}{#2}
}
\providecommand{\href}[2]{#2}
\begin{thebibliography}{10}

\bibitem{AdamsFranzosa:2007}
C.~Adams and R.~Franzosa, \emph{Introduction to topology: Pure and applied, 1st
  ed.}, Pearson, London, UK, 2008, 512 pp.,ISBN-13: 9780131848696.

\bibitem{Brouwer1911fixedPoints}
L.E.J. Brouwer, \emph{\"{U}ber abbildung von mannigfaltigkeiten}, Math. Ann.
  \textbf{71} (1911), 97--115.

\bibitem{DiConcilio2018MCSdescriptiveProximities}
A.~Di Concilio, C.~Guadagni, J.F. Peters, and S.~Ramanna, \emph{Descriptive
  proximities. properties and interplay between classical proximities and
  overlap}, Math. Comput. Sci. \textbf{12} (2018), no.~1, 91--106, MR3767897,
  Zbl 06972895.

\bibitem{Day1957amenableSemigroups}
M.M. Day, \emph{Amenable semigroups}, Illinois J. Math. \textbf{1} (1957),
  509--544,
  MR0092128,MR0044031,\url{https://projecteuclid.org/download/pdf_1/euclid.ijm/1255380675}.

\bibitem{Day1961amenableSemigroups}
\bysame, \emph{Fixed-point theorems for compact convex sets}, Illinois J. Math.
  \textbf{5} (1961), 585--590,
  MR0138100,MR0092128,\url{https://projecteuclid-org.uml.idm.oclc.org/download/pdf_1/euclid.ijm/1255631582},
  reviewed by M. Edelstein.

\bibitem{FrishTamuzFerdowsi2019strongAmenability}
J.~Frisch, O.~Tamuz, and P.V. Ferdowsi, \emph{Strong amenability and the
  infinite conjugacy class property}, Invent. Math. \textbf{218} (2019), no.~3,
  833--351, MR4022081,\url{https://doi.org/10.1007/s00222-019-00896-z}.

\bibitem{Kakutani1938twoFixedPointTheorems}
S.~Kakutani, \emph{Two fixed-point theorems concerning bicompact convex sets},
  Proc. Imp. Acad. Tokyo \textbf{14} (1938), no.~7, 242--245, MR1568507.

\bibitem{Markov1936someTheorems}
A.A. Markov, \emph{Quelques th\'{e}r\`{e} sur les ensembles ab\'{e}liens}, C.R.
  (Doklady) Acad. Sci. URSS (N.S.) \textbf{1} (1936), 311--313.

\bibitem{Naimpally2013}
S.A. Naimpally and J.F. Peters, \emph{Topology with applications. {T}opological
  spaces via near and far}, World Scientific, Singapore, 2013, xv + 277 pp,
  Amer. Math. Soc. MR3075111.

\bibitem{Naimpally70}
S.A. Naimpally and B.D. Warrack, \emph{Proximity spaces}, Cambridge Tract in
  Mathematics No. 59, Cambridge University Press, Cambridge, UK, 1970, x+128
  pp.,Paperback (2008),MR0278261.

\bibitem{Peters2013springer}
J.F. Peters, \emph{Topology of digital images. {V}isual pattern discovery in
  proximity spaces}, Intelligent Systems Reference Library, vol.~63, Springer,
  2014, xv + 411pp, Zentralblatt MATH Zbl 1295 68010.

\bibitem{Peters2019vortexNerves}
\bysame, \emph{Vortex nerves and their proximities. nerve betti numbers and
  descriptive proximity}, Bull. Allahabad Math. Soc. \textbf{34} (2019), no.~2,
  263--276, \url{https://arxiv.org/abs/1910.08467}, Zbl 07178002.

\bibitem{Peters2020AMSBullRibbonComplexes}
\bysame, \emph{Ribbon complexes\ \&\ their approximate descriptive proximities.
  {R}ibbon\ \&\ vortex nerves, betti numbers and planar divisions}, Bull.
  Allahabad Math. Soc. \textbf{35} (2020), 1--13, \emph{in press}, preprint:
  \url{https://arxiv.org/abs/1911.09014} [math.GT].

\bibitem{Spanier1966AlgTopology}
E.H. Spanier, \emph{Algebraic topology}, McGraw-Hill Book Co., New
  York-Toronto, Ont.,CA, 1966, {x}iv+528 pp.,MR0210112.

\bibitem{Cech1966}
E.~\u{C}ech, \emph{Topological spaces}, John Wiley \& Sons Ltd., London, 1966,
  fr seminar, Brno, 1936-1939; rev. ed. Z. Frolik, M. Kat\u{e}tov.

\bibitem{Whitehead1939homotopy}
J.H.C. Whitehead, \emph{Simplicial spaces, nuclei and m-groups}, Proceedings of
  the London Math. Soc. \textbf{45} (1939), 243--327.

\bibitem{Whitehead1949BAMS-CWtopology}
\bysame, \emph{Combinatorial homotopy. {I}}, Bulletin of the American
  Mathematical Society \textbf{55} (1949), no.~3, 213--245, Part 1.

\end{thebibliography}

\end{document}